\documentclass[12pt]{amsart}
\usepackage{amssymb}
\usepackage{amsfonts}
\usepackage{graphicx}
\usepackage{amscd}

\newtheorem{theorem}{Theorem}

\newtheorem{example}{Example}

\newtheorem{lemma}{Lemma}

\newtheorem{remark}{Remark}

\numberwithin{equation}{section}

\begin{document}
\title[Hermite-Hadamard type inequalities]{Some new Hermite-Hadamard type inequalities
for two operator convex functions}
\author[A. G. Ghazanfari]{ Amir G. Ghazanfari}
\address{Department of Mathematics\\
Lorestan University, P.O. Box 465\\
Khoramabad, Iran.}
\email{ghazanfari.amir@gmail.com}

\begin{abstract}\noindent
In this paper we establish some new Hermite-Hadamard type inequalities for two operator convex functions of
selfadjoint operators in Hilbert spaces.
\end{abstract}
\maketitle

\section{introduction}\vspace{.2cm} \noindent
The following inequality holds for any convex function $f$ defined on $\mathbb{R}$ and $a,b \in \mathbb{R}$, with $a<b$
\begin{equation}\label{1.1}
f\left(\frac{a+b}{2}\right)\leq \frac{1}{b-a}\int_a^b f(x)dx\leq \frac{f(a)+f(b)}{2}
\end{equation}
Both inequalities hold in the reversed direction if $f$ is concave. We note that
Hermite-Hadamard's inequality may be regarded as a refinement of the concept of convexity and it follows easily from
Jensen's inequality.
The classical Hermite-Hadamard inequality provides estimates of the mean value
of a continuous convex function $f : [a, b] \rightarrow \mathbb{R}$.

Let $X$ be a vector space, $ x, y \in X$, $ x\neq y$. Define the segment
\[
[x, y] :=\{(1- t)x + ty~ :~ t \in [0, 1]\}.
\]
We consider the function $f : [x, y]\rightarrow \mathbb{R}$ and the associated function
\begin{align*}
&g(x, y) : [0, 1] \rightarrow \mathbb{R},\\
&g(x, y)(t) := f((1 - t)x + ty), t \in [0, 1].
\end{align*}
Note that f is convex on $[x, y]$ if and only if $g(x, y)$ is convex on $[0, 1]$.
For any convex function defined on a segment $[x, y] \subseteq X$, we have the Hermite-Hadamard integral inequality
\begin{equation}\label{1.2}
f\left(\frac{x+y}{2}\right)\leq \int_0^1 f((1-t)x+ty)dt\leq \frac{f(x)+f(y)}{2},
\end{equation}
which can be derived from the classical Hermite-Hadamard inequality (1.1) for the
convex function $g(x, y) : [0; 1] \rightarrow \mathbb{R}$.

In recent years several extensions and generalizations have been considered for
classical convexity. A number of papers have been written on this inequality providing
some inequalities analogous to Hadamard's inequality
given in (\ref{1.1}) involving two convex functions, see \cite{pach, tun}.
The main purpose of this paper is to establish some new integral inequalities
for two operator convex functions of
selfadjoint operators in Hilbert spaces which generalize Theorem 1 in \cite{pach} and Theorem 3 in \cite{tun}.

In order to do that we need the following preliminary definitions and results.
Let $A$ be a bounded self adjoint linear operator on a complex Hilbert space $(H; \langle .,.\rangle)$.
The Gelfand map establishes a $*$-isometrically isomorphism $\Phi$ between the set
$C (Sp (A))$ of all continuous functions defined on the spectrum of $A$, denoted $Sp (A)$,
and the $C^*$-algebra $C^*(A)$ generated by $A$  and the identity operator $1_H$ on $H$ as
follows (see for instance \cite[p.3]{fur}).
For any $f, g\in C (Sp (A))$ and any $\alpha,\beta\in\mathbb{C}$ we have
\begin{align*}
&(i)~\Phi(\alpha f+\beta g)=\alpha\Phi(f)+\beta\Phi(g);\\
&(ii)~\Phi(fg)=\Phi(f)\Phi(g)~\text{and}~\Phi(f^*)=\Phi(f)^*;\\
&(iii)~\|\Phi(f)\|=\|f\|:=\sup_{t\in Sp(A)} |f(t)|;\\
&(iv)~\Phi(f_0)=1 ~\text{and}~ \Phi(f_1)=A,~\text{ where } f_0(t)=1~ \text{and } f_1(t)=t, \text{for}~ t\in Sp(A).
\end{align*}
With this notation we define
\[ f (A) := \Phi(f) \text{ for all } f \in C (Sp (A))\]
and we call it the continuous functional calculus for a bounded selfadjoint operator $A$.
If $A$ is a bounded selfadjoint operator and $f$ is a real valued continuous function on $Sp (A)$,
then $f (t) \geq 0$ for any $t \in Sp (A)$ implies that $f (A) \geq 0$, i.e., $f (A)$ is a positive
operator on $H$. Moreover, if both $f$ and $g$ are real valued functions on $Sp (A)$ then
the following important property holds:
\begin{align*}
&(P) &&f (t) \geq g (t) \text{ for any } t \in Sp (A) \text{ implies that } f (A)\geq g (A)
\end{align*}
in the operator order in $B(H)$.

A real valued continuous function $f$ on an interval $I$ is said to be operator convex
(operator concave) if
\begin{align*}
&(OC)\quad &&f ((1 -\lambda)A +\lambda B)\leq (\geq) (1-\lambda) f (A) + \lambda f (B)
\end{align*}
in the operator order in $B(H)$, for all $ \lambda\in [0, 1] $ and for every bounded selfadjoint operators $A$ and $B$
in $B(H)$ whose spectra are contained in $I$.

Dragomir in \cite{dra2} has proved a Hermite-Hadamard type inequality for operator convex function as follows:

\begin{theorem}\label{t1}
Let $f : I \rightarrow \mathbb{R}$ be an operator convex function on the interval $I$. Then
for any selfadjoint operators $A$ and $B$ with spectra in $I$, the following inequality holds
\begin{multline*}
\left(f\left(\frac{A+B}{2}\right)\leq\right)\frac{1}{2}\left[f\left(\frac{3A+B}{4}\right)+f\left(\frac{A+3B}{4}\right)\right]\\
\leq\int_0^1 f((1-t)A+tB))dt\\
\leq\frac{1}{2}\left[f\left(\frac{A+B}{2}\right)+\frac{f(A)+f(B)}{2}\right]\left(\leq\frac{f(A)+f(B)}{2}\right).
\end{multline*}
\end{theorem}
A generalization of Theorem \ref{t1} for operator preinvex
functions of selfadjoint operators in Hilbert spaces is established in \cite{gha}, as follows:
\begin{theorem}\label{t2}
Let $S\subseteq B(H)_{sa}$ be an invex set with respect to $\eta:S\times S\rightarrow B(H)_{sa}$ and $\eta$ satisfies condition $C$. If for every $A,B\in S$ and $V=A+\eta(B,A)$ the function $f : I \rightarrow \mathbb{R}$ is operator preinvex with respect to $\eta$ on $\eta-$path $P_{AV}$ with spectra of $A$ and spectra of $V$ in the interval $I$. Then the following inequality holds
\begin{align}\label{1.3}
f\left(\frac{A+V}{2}\right)&\leq \frac{1}{2}\left[ f\left(\frac{3A+V}{4}\right)+f\left(\frac{A+3V}{4}\right)\right]\notag\\
&\leq \int_0^1 f(A+t\eta(B,A))dt\\
&\leq\frac{1}{2}\left[ f\left(\frac{A+V}{2}\right)+\frac{f(A)+f(V)}{2}\right]\leq \frac{f(A)+f(B)}{2}\notag.
\end{align}
\end{theorem}

\section{the main results}

Let $f,g : I \rightarrow \mathbb{R}$ be operator convex functions on the interval $I$. Then
for any selfadjoint operators $A$ and $B$ on a Hilbert space $H$ with spectra in $I$, we define real functions
$M(A,B)$, $N(A,B)$ and $P(A,B)$ on $H$ by
\begin{align*}
M(A,B)(x)&=\langle f(A)x,x\rangle\langle g(A)x,x\rangle+\langle f(B)x,x\rangle\langle g(B)x,x\rangle\quad &&(x\in H),\\
N(A,B)(x)&=\langle f(A)x,x\rangle\langle g(B)x,x\rangle+\langle f(B)x,x\rangle\langle g(A)x,x\rangle\quad &&(x\in H),\\
P(A,B)(x)&=\langle [f(A)g(A)+f(B)g(B)]x,x\rangle\quad &&(x\in H).
\end{align*}
\begin{lemma}
Let $f:I\rightarrow \Bbb R$ be a continuous function on the interval $I$. Then for every two selfadjoint operators $A,B$ with spectra
in $I$ the function $f$ is operator convex on $[A,B]$ if and only if the function $\varphi_{x,A,B}:[0,1]\rightarrow \Bbb R$ defined by
\begin{equation}\label{2.1}
\varphi_{x,A,B}=\langle f((1-t)A+tB) x,x \rangle
\end{equation}
is convex on $[0,1]$ for every $x\in H$ with $\|x\|=1$.
\end{lemma}

\begin{proof}
Let $f$ be operator convex on $[A,B]$ then for any $t_{1} , t_{2} \in [0,1] $ and $\alpha , \beta \geq 0 $ with $\alpha + \beta = 1$ we have %
\begin{multline*}
\varphi_{x,A,B} ( \alpha t_{1} + \beta t_{2} ) =
 \langle f((1-(\alpha t_{1} +\beta t_{2} ))A +(\alpha t_{1} +\beta t_{2} )B)x,x \rangle\\
=\langle f(\alpha [(1-t_{1})A+t_{1} B]+\beta [(1-t_{2} )A+t_{2} B])x,x \rangle\\
\leq \alpha \langle f((1-t_{1} )A+t_{1} B)x,x\rangle + \beta \langle f((1-t_{2})A +t_{2} B)x,x\rangle \\
= \alpha \varphi_{x,A,B} (t_{1} )+\beta \varphi_{x,A,B}(t_2).
\end{multline*}
showing that $\varphi_{x,A,B}$ is a convex function on $[0, 1]$.

Let now $\varphi_{x,A,B}$ be convex on $[0,1]$, we show that $f$ is operator convex on $[A,B]$.
For every  $C=(1-t_{1})A+t_{1}B$ and $D=(1-t_{2})A+t_{2}B$ we have
\begin{multline*}
\langle f(\lambda C+(1-\lambda)D)x,x\rangle\\
=\langle f[\lambda((1-t_{1})A+t_{1}B)+(1-\lambda)((1-t_{2})A+t_{2}B)]x,x\rangle\\
=\langle f[(1-(\lambda t_{1}+(1-\lambda)t_{2}))A+(\lambda t_{1}+(1-\lambda)t_{2})B]x,x\rangle\\
=\varphi_{x,A,B}(\lambda t_{1}+(1-\lambda)t_{2})
\leq\lambda \varphi_{x,A,B}(t_{1})+(1-\lambda)\varphi_{x,A,B}(t_{2})\\
=\lambda \langle f((1-t_{1})A+t_{1}B)x,x\rangle+(1-\lambda)\langle f((1-t_{2})A+t_{2}B)x,x\rangle\\
\leq\lambda \langle f(C)x,x\rangle+(1-\lambda)\langle f(D)x,x\rangle.
\end{multline*}
\end{proof}

\begin{theorem}\label{t3}
Let $f,g : I \rightarrow \mathbb{R^+}$ be operator convex functions on the interval $I$. Then
for any selfadjoint operators $A$ and $B$ on a Hilbert space $H$ with spectra in $I$, the inequality

\begin{multline}\label{2.2}
\int_{0}^{1}\langle f(tA+(1-t)B)x,x\rangle\langle g(tA+(1-t)B)x,x\rangle dt\\
\leq\frac{1}{3}M(A,B)(x)+\frac{1}{6}N(A,B)(x).
\end{multline}

holds for any $x\in H$ with $\|x\| = 1$.
\end{theorem}

\begin{proof}
For $x\in H$ with $\|x\|=1$ and $t\in [0,1]$, we have
\begin{equation}\label{2.3}
\langle \big(tA+(1-t)B\big)x,x\rangle =t\langle Ax,x\rangle+(1-t)\langle Bx,x\rangle \in I,
\end{equation}
since $\langle Ax,x\rangle\in Sp(A)\subseteq I$ and $\langle Bx,x\rangle\in Sp(B)\subseteq I$.

Continuity of $f,g$ and (\ref{2.3}) imply that the operator valued integrals $\int_0^1 f(tA+(1-t)B)dt$,
 $~\int_0^1 g(tA+(1-t)B)dt$ and $\int_0^1 (fg)(tA+(1-t)B)dt$ exist.

Since $f$ and $g$ are operator convex, therefore for $t$ in $[0,1]$ and $x\in H$ we have
\begin{equation}\label{2.4}
\langle f(tA+(1-t)B)x,x\rangle\leq\langle (tf(A)+(1-t)f(B))x,x\rangle,
\end{equation}
\begin{equation}\label{2.5}
\langle g(tA+(1-t)B)x,x\rangle\leq\langle(tg(A)+(1-t)g(B))x,x\rangle.
\end{equation}

From (\ref{2.4}) and (\ref{2.5}) we obtain

\begin{multline}\label{2.6}
\langle f(tA+(1-t)B)x,x\rangle\langle g(tA+(1-t)B)x,x\rangle\\
\leq t^2\langle f(A)x,x\rangle\langle g(A)x,x\rangle+(1-t)^2\langle f(B)x,x\rangle\langle g(B)x,x\rangle\\
+t(1-t)\left[\langle f(A)x,x\rangle\langle g(B)x,x\rangle+\langle f(B)x,x\rangle \langle g(A)x,x\rangle\right].
\end{multline}

Integrating both sides of (\ref{2.6}) over $[0,1]$ we get the required inequality (\ref{2.2}).
\end{proof}

\begin{theorem}\label{t4}
Let $f,g : I \rightarrow \mathbb{R}$ be operator convex functions on the interval $I$. Then
for any selfadjoint operators $A$ and $B$ on a Hilbert space $H$ with spectra in $I$, the inequality

\begin{multline}\label{2.7}
\left\langle f\left(\frac{A+B}{2}\right)x,x\right\rangle\left\langle g\left(\frac{A+B}{2}\right)x,x\right\rangle\\
\leq\frac{1}{2}\int_{0}^{1}\left\langle f(tA+(1-t)B)x,x\rangle\langle g(tA+(1-t)B)x,x\right\rangle dt \\
+\frac{1}{12}M(A,B)(x)+\frac{1}{6}N(A,B)(x),
\end{multline}

holds for any $x\in H$ with $\|x\| = 1$.
\end{theorem}

\begin{proof}
Since $f$ and $g$ are operator convex, therefore for any $t\in I$ and any $x\in H$ with $\|x\| = 1$ we observe that
\begin{multline*}
\left\langle f\left(\frac{A+B}{2}\right)x,x\right\rangle \left\langle g\left(\frac{A+B}{2}\right)x,x\right\rangle\\
=\left\langle f\left(\frac{tA+(1-t)B}{2}+\frac{(1-t)A+tB}{2}\right)x,x\right\rangle\\
 \times\left\langle g\left(\frac{tA+(1-t)B}{2}+\frac{(1-t)A+tB}{2}\right)x,x\right\rangle
\end{multline*}
\begin{multline*}
\leq\frac{1}{4}[\langle f(tA+(1-t)B)x,x\rangle +\langle f(1-t)A+tB)x,x\rangle]\\
\times[\langle g(tA+(1-t)B)x,x\rangle+\langle g((1-t)A+tB)x,x\rangle]
\end{multline*}
\begin{multline*}
\leq\frac{1}{4}[\langle f(tA+(1-t)B)x,x\rangle \langle g(tA+(1-t)B)x,x\rangle\\
+\langle f((1-t)A+tB)x,x\rangle \langle g((1-t)A+tB)x,x\rangle]\\
+\frac{1}{4}[t\langle f(A)x,x\rangle+(1-t)\langle f(B)x,x\rangle][(1-t)\langle g(A)x,x\rangle+t\langle g(B)x,x\rangle]\\
+[(1-t)\langle f(A)x,x\rangle+t\langle f(B)x,x\rangle][t\langle g(A)x,x\rangle+(1-t)\langle g(B)x,x\rangle]
\end{multline*}
\begin{multline*}
=\frac{1}{4}[\langle f(tA+(1-t)B)x,x\rangle \langle g(tA+(1-t)B)x,x\rangle\\
+\langle f((1-t)A+tB)x,x\rangle\langle g((1-t)A+tB)x,x\rangle]\\
+\frac{1}{4}2t(1-t)[\langle f(A)x,x\rangle\langle g(A)x,x\rangle+\langle f(B)x,x\rangle \langle g(B)x,x\rangle]\\
+(t^2+(1-t)^2)[\langle f(A)x,x\rangle \langle g(B)x,x\rangle+\langle f(B)x,x\rangle \langle g(A)x,x\rangle].
\end{multline*}
 Therefore we get
 \begin{multline}\label{2.8}
\left\langle f\left(\frac{A+B}{2}\right)x,x\right\rangle \left\langle g\left(\frac{A+B}{2}\right)x,x\right\rangle\\
\leq\frac{1}{4}[\langle f(tA+(1-t)B)x,x\rangle \langle g(tA+(1-t)B)x,x\rangle\\
+\langle f((1-t)A+tB)x,x\rangle\langle g((1-t)A+tB)x,x\rangle]\\
+\frac{1}{4}2t(1-t)[\langle f(A)x,x\rangle\langle g(A)x,x\rangle+\langle f(B)x,x\rangle \langle g(B)x,x\rangle]\\
+(t^2+(1-t)^2)[\langle f(A)x,x\rangle \langle g(B)x,x\rangle+\langle f(B)x,x\rangle \langle g(A)x,x\rangle].
\end{multline}

We integrate both sides of (\ref{2.8}) over [0,1] and obtain
\begin{multline*}
\left\langle f\left(\frac{A+B}{2}\right)x,x\right\rangle \left\langle g\left(\frac{A+B}{2}\right)x,x\right\rangle\\
\leq\frac{1}{4}\int_{0}^{1}[\langle f(tA+(1-t)B)x,x\rangle\langle g(tA+(1-t)B)x,x\rangle\\
+\langle f((1-t)A+tB)x,x\rangle\langle g((1-t)A+tB)x,x\rangle]dt\\
+\frac{1}{12}M(A,B)(x)+\frac{1}{6}N(A,B)(x).
\end{multline*}

This implies the required inequality (\ref{2.7}).
\end{proof}

\begin{theorem}\label{t5}
Let $f,g : I \rightarrow \mathbb{R}$ be operator convex functions on the interval $I$. Then
for any selfadjoint operators $A$ and $B$ on a Hilbert space $H$ with spectra in $I$, we have the inequality
\begin{multline}\label{2.9}
\left\langle f\left(\frac{A+B}{2}\right)x,x\right\rangle \int_{0}^{1}\langle g(tA+(1-t)B)x,x\rangle dt\\
+\left\langle g\left(\frac{A+B}{2}\right)x,x\right\rangle \int_{0}^{1}\langle f(tA+(1-t)B)x,x\rangle dt\\
\leq\frac{1}{2}\int_{0}^{1}\langle f(tA+(1-t)B)x,x\rangle\langle g(tA+(1-t)B)x,x\rangle dt\\
+\frac{1}{12} M(A,B)(x)
+\frac{1}{6} N(A,B)(x)
+\left\langle f\left(\frac{A+B}{2}\right)x,x\right\rangle\left\langle g\left(\frac{A+B}{2}\right)x,x\right\rangle.
\end{multline}
\end{theorem}

\begin{proof}
Since $f$ and $g$ are operator convex, then for $t\in[0,1]$ we observe that
\begin{align*}
\left\langle f\left(\frac{A+B}{2}\right)x,x\right\rangle&=\left\langle f\left(\frac{tA+(1-t)B}{2}+\frac{(1-t)A+tB}{2}\right)x,x\right\rangle\\
&\leq\left\langle\frac{f(tA+(1-t)B)+f((1-t)A+tB)}{2}x,x\right\rangle\\
\left\langle g\left(\frac{A+B}{2}\right)x,x\right\rangle&=\left\langle g\left(\frac{tA+(1-t)B}{2}+\frac{(1-t)A+tB}{2}\right)x,x\right\rangle\\
&\leq\left\langle\frac{g(tA+(1-t)B)+g((1-t)A+tB)}{2}x,x\right\rangle.
\end{align*}
we multiply by one under the
other and by one across the other of the above inequality and then we add these inequalities, so we obtain
\begin{multline*}
\left\langle f\left(\frac{A+B}{2}\right)x,x\right\rangle\left\langle\frac{g(tA+(1-t)B)+g((1-t)A+tB)}{2}x,x\right\rangle\\
+\left\langle g\left(\frac{A+B}{2}\right)x,x\right\rangle\left\langle\frac{f(tA+(1-t)B)+f((1-t)A+tB)}{2}x,x\right\rangle\\
\leq\left\langle\frac{f(tA+(1-t)B)+f((1-t)A+tB)}{2}x,x\right\rangle\\
\quad \quad\quad\times\left\langle\frac{g(tA+(1-t)B)+g((1-t)A+tB)}{2}x,x\right\rangle\\
\quad\quad\quad\quad\quad+\left\langle f\left(\frac{A+B}{2}\right)x,x\right\rangle \left\langle g\left(\frac{A+B}{2}\right)x,x\right\rangle.
\end{multline*}
This implies that
\begin{align*}
&\frac{1}{2}\left\langle f\left(\frac{A+B}{2}\right)x,x\right\rangle[\langle g(tA+(1-t)B)x,x\rangle+\langle g((1-t)A+tB)x,x\rangle]\\
&+\frac{1}{2}\left\langle g\left(\frac{A+B}{2}\right)x,x\right\rangle[\langle f(tA+(1-t)B)x,x\rangle+\langle f((1-t)A+tB)x,x\rangle]\\
&\leq\frac{1}{4}[\langle f(tA+(1-t)B)x,x\rangle \langle g(tA+(1-t)B)x,x\rangle\\
&+\langle f((1-t)A+tB)x,x\rangle \langle g((1-t)A+tB)x,x\rangle]\\
&+\frac{1}{4}[t\langle f(A)x,x\rangle+(1-t)\langle f(B)x,x\rangle][(1-t)\langle g(A)x,x\rangle+t\langle g(B)x,x\rangle]\\
&+\frac{1}{4}[(1-t)\langle f(A)x,x\rangle+t\langle f(B)x,x\rangle][t\langle g(A)x,x\rangle+(1-t)\langle g(B)x,x\rangle]\\
&+\left\langle f\left(\frac{A+B}{2}\right)x,x\right\rangle \left\langle g\left(\frac{A+B}{2}\right)x,x\right\rangle.
\end{align*}
Again integration both side of the above inequality over $[0,1]$ and obtain\\
\begin{align*}
&\frac{1}{2}\left\langle f\left(\frac{A+B}{2}\right)x,x\right\rangle\int_0^1[\langle g(tA+(1-t)B)x,x\rangle+\langle g((1-t)A+tB)x,x\rangle]dt\\
&+\frac{1}{2}\left\langle g\left(\frac{A+B}{2}\right)x,x\right\rangle\int_0^1[\langle f(tA+(1-t)B)x,x\rangle+\langle f((1-t)A+tB)x,x\rangle]dt\\
&\leq\frac{1}{4}\int_0^1[\langle f(tA+(1-t)B)x,x\rangle \langle g(tA+(1-t)B)x,x\rangle\\
&\quad\quad\quad\quad\quad\quad+\langle f((1-t)A+tB)x,x\rangle \langle g((1-t)A+tB)x,x\rangle]dt\\
&+\frac{1}{4}M(A,B)(x)\int_{0}^{1}[2t(1-t)]dt+\frac{1}{4}N(A,B)(x)\int_{0}^{1}[t^2+(1-t)^2]dt\\
&+\left\langle f\left(\frac{A+B}{2}\right)x,x\right\rangle \left\langle g\left(\frac{A+B}{2}\right)x,x\right\rangle\int_{0}^{1}dt.
\end{align*}
Then we have the required inequality (\ref{2.9})
\end{proof}

\section{application for synchronous (asynchronous) functions}

We say that the functions $f, g : [a, b]\rightarrow \mathbb{R}$ are synchronous (asynchronous)
on the interval $[a, b]$ if they satisfy the following condition:
$$ (f (t)-f (s)) (g (t)-g (s)) \geq (\leq)~ 0 \text{ for each } t, s\in [a, b]$$
It is obvious that, if $f, g$ are monotonic and have the same monotonicity
on the interval $[a, b]$, then they are synchronous on $[a, b]$ while if they have
opposite monotonicity, they are asynchronous.

The following result provides an inequality of $\check{C}eby\check{s}ev$ type for functions
of selfadjoint operators.
\begin{theorem}\label{t6}{\rm (Dragomir,\cite{dra1})} Let $A$ be a selfadjoint operator
with $Sp (A) \subset [m,M]$ for some real numbers $m < M$, If $f, g : [m,M]\rightarrow \mathbb{R}$
are continuous and synchronous (asynchronous) on $[m,M]$, then
\begin{equation}\label{3.1}
\langle f (A) g (A) x, x\rangle \geq (\leq) \langle f (A) x, x\rangle \langle g (A) x, x\rangle,
\end{equation}
for any $x \in H$ with $\|x\| = 1$.
\end{theorem}
As a simple consequence of the above Theorem we imply that if $f,g$ are synchronous, then
\begin{equation}\label{3.2}
N(A,B)(x)\leq  M(A,B)(x)\leq P(A,B)(x),
\end{equation}
for any $x \in H$ with $\|x\| = 1$. If $f,g$ are asynchronous, then reverse inequalities holds in (\ref{3.2}) as follow,
\begin{equation}\label{3.3}
N(A,B)(x)\geq  M(A,B)(x)\geq P(A,B)(x).
\end{equation}
\begin{remark}\label{r1}
\begin{enumerate}
\rm {Let $f, g : [m,M]\rightarrow \mathbb{R}$ be operator convex and $A,B$ be selfadjoint operator with
$Sp (A)\cup sp(B) \subset [m,M]$
\item[(i)] If $f,g$ are synchronous and $f,g\geq 0$ then the inequality (\ref{2.2}) becomes
\begin{multline}\label{3.4}
\int_{0}^{1}\langle f(tA+(1-t)B)x,x\rangle\langle g(tA+(1-t)B)x,x\rangle dt\\
\leq\frac{1}{3}M(A,B)(x)+\frac{1}{6}N(A,B)(x)
\leq\frac{1}{2}P(A,B)(x).
\end{multline}
If $f,g$ are synchronous then the inequalities (\ref{2.7}) and (\ref{2.9}) become
\begin{multline}\label{3.5}
\left\langle f\left(\frac{A+B}{2}\right)x,x\right\rangle\left\langle g\left(\frac{A+B}{2}\right)x,x\right\rangle\\
\leq\frac{1}{2}\int_{0}^{1}\left\langle f(tA+(1-t)B) g(tA+(1-t)B)x,x\right\rangle dt \\
+\frac{1}{4}P(A,B)(x),
\end{multline}
and
\begin{multline}\label{3.6}
\left\langle f\left(\frac{A+B}{2}\right)x,x\right\rangle \int_{0}^{1}\langle g(tA+(1-t)B)x,x\rangle dt\\
+\left\langle g\left(\frac{A+B}{2}\right)x,x\right\rangle \int_{0}^{1}\langle f(tA+(1-t)B)x,x\rangle dt\\
\leq\frac{1}{2}\int_{0}^{1}\langle f(tA+(1-t)B) g(tA+(1-t)B)x,x\rangle dt\\
+\frac{1}{4} P(A,B)(x)
+\left\langle f\left(\frac{A+B}{2}\right) g\left(\frac{A+B}{2}\right)x,x\right\rangle.
\end{multline}
\item[(ii)] If $f,g$ are asynchronous and $f,g\geq 0$ then the inequality (\ref{2.2}) becomes
\begin{multline}\label{3.7}
\int_{0}^{1}\langle f(tA+(1-t)B) g(tA+(1-t)B)x,x\rangle dt\\
\leq\frac{1}{3}M(A,B)(x)+\frac{1}{6}N(A,B)(x)
\leq\frac{1}{2}N(A,B)(x).
\end{multline}
If $f,g$ are synchronous then the inequalities (\ref{2.7}) and (\ref{2.9}) become
\begin{multline}\label{3.8}
\left\langle f\left(\frac{A+B}{2}\right) g\left(\frac{A+B}{2}\right)x,x\right\rangle\\
\leq\frac{1}{2}\int_{0}^{1}\left\langle f(tA+(1-t)B)x,x\rangle\langle g(tA+(1-t)B)x,x\right\rangle dt \\
+\frac{1}{4}N(A,B)(x)
\end{multline}
and
\begin{multline}\label{3.9}
\left\langle \left[f\left(\frac{A+B}{2}\right) \int_{0}^{1} g(tA+(1-t)B) dt\right.\right.\\
\left.\left.+ g\left(\frac{A+B}{2}\right) \int_{0}^{1} f(tA+(1-t)B)dt\right]x,x\right\rangle \\
\leq\frac{1}{2}\int_{0}^{1}\langle f(tA+(1-t)B)x,x\rangle\langle g(tA+(1-t)B)x,x\rangle dt\\
+\frac{1}{4} N(A,B)(x)
+\left\langle f\left(\frac{A+B}{2}\right)x,x\right\rangle\left\langle g\left(\frac{A+B}{2}\right)x,x\right\rangle.
\end{multline}
}
\end{enumerate}
\end{remark}
\begin{example}
Real functions $f(x)=x$ and $g(x)=x^2$ are synchronous on $[0,1]$. From inequality (\ref{3.4})
we get the inequality
\begin{equation*}
\int_0^1\langle(tA+(1-t)B)x,x\rangle\langle (tA+(1-t)B)^2 x,x\rangle dt\leq\frac{1}{2}P(A,B)(x),
\end{equation*}
which holds for any selfadjoint operators $A$ and $B$ on a Hilbert space $H$ with spectra in $[0,1]$.

Real functions $f,g$ are asynchronous on $[-1,0]$, Now from inequality (\ref{3.8})
we get the inequality
\begin{multline*}
\left\langle\left(\frac{A+B}{2}\right)^3 x,x\right\rangle\\
\leq\frac{1}{2}\int_{0}^{1}\left\langle (tA+(1-t)B)x,x\rangle\langle (tA+(1-t)B)^2x,x\right\rangle dt \\
+\frac{1}{4}N(A,B)(x)
\end{multline*}

which holds for any selfadjoint operators $A$ and $B$ on a Hilbert space $H$ with spectra in $[-1,0]$, where
\begin{align*}
M(A,B)(x)&=\langle Ax,x\rangle\langle A^2x,x\rangle+\langle Bx,x\rangle\langle B^2x,x\rangle,\\
N(A,B)(x)&=\langle Ax,x\rangle\langle B^2x,x\rangle+\langle Bx,x\rangle\langle A^2x,x\rangle,\\
P(A,B)(x)&=\langle (A^3+B^3)x,x\rangle.
\end{align*}
We may obtain other inequalities which follow from (\ref{3.5}), (\ref{3.6}), (\ref{3.7}) and (\ref{3.9}), the details are omitted.
\end{example}

\end{document}